\documentclass[final,a4paper,12pt]{amsart}
\usepackage[T1]{fontenc}
\usepackage{ae,aecompl}
\usepackage[utf8]{inputenc}
\usepackage{amssymb}
\usepackage{enumitem}
\usepackage{hyperref}
\usepackage{xcolor}
\hypersetup{
	colorlinks,
	linkcolor={red!80!black},
	citecolor={blue!50!black},
	urlcolor={blue!80!black}
}
\usepackage{url}
\usepackage{tikz-cd}
\usepackage{float}
\usepackage{caption}
\usepackage[british]{babel}
\usepackage{csquotes}
\usepackage{needspace}
\usepackage[backend=biber,
maxnames=99,
maxalphanames=4,
url=false,
isbn=true,
backref=true,
citestyle=alphabetic,
bibstyle=alphabetic,
autocite=inline,
sorting=nty,]{biblatex}

\newtheorem*{mainthm}{Main Theorem}
\newtheorem{thm}{Theorem}[section]
\newtheorem{lem}[thm]{Lemma}
\newtheorem{cor}[thm]{Corollary}
\newtheorem{prop}[thm]{Proposition}
\newtheorem{fct}[thm]{Fact}

\newtheorem{qu}[thm]{Question}
\theoremstyle{remark}
\newtheorem{rem}[thm]{Remark}
\theoremstyle{definition}
\newtheorem{dfn}[thm]{Definition}
\newtheorem{ex}[thm]{Example}
\newtheorem*{clm*}{Claim}
\newenvironment{clmproof}[1][\proofname]{\proof[#1]}{\endproof}

\DeclareMathOperator{\Stab}{{Stab}}

\DeclareMathOperator{\tp}{{tp}}
\DeclareMathOperator{\Aut}{{Aut}}

\DeclareMathOperator{\Gal}{{Gal}}
\DeclareMathOperator{\SO}{{SO}}
\DeclareMathOperator{\GL}{{GL}}
\newcommand{\fC}{\mathfrak C}

\newcommand{\Autf}{\operatorname{Aut\,f}}

\newcommand{\restr}{\mathord{\upharpoonright}}

\newcommand{\Er}{\mathrel{E}}
\newcommand{\Rr}{\mathrel{R}}

\let\Gamma\varGamma
\let\Delta\varDelta
\let\Theta\varTheta
\let\Lambda\varLambda
\let\Xi\varXi
\let\Pi\varPi
\let\Sigma\varSigma
\let\Upsilon\varUpsilon
\let\Phi\varPhi
\let\Psi\varPsi
\let\Omega\varOmega
\let\phi\varphi

\newcommand{\xqed}[1]{%
	\leavevmode\unskip\penalty9999 \hbox{}\nobreak\hfill
	\quad\hbox{\ensuremath{#1}}}

\begin{document}
	
	
	\author{Tomasz Rzepecki}
	\email{tomasz.rzepecki@math.uni.wroc.pl}
	\address{
		Instytut Matematyczny, Uniwersytet Wrocławski\\
		pl. Grunwaldzki 2/4\\
		50-384 Wrocław, Poland
	}
	\thanks{The author is supported by NCN grant 2015/17/N/ST1/02322}

	
	\title{Equivalence relations invariant under group actions}
	
	\begin{abstract}
		We extend some recent results about bounded invariant equivalence relations and invariant subgroups of definable groups: we show that type-definability and smoothness are equivalent conditions in a wider class of relations than heretofore considered, which includes all the cases for which the equivalence was proved before.
		
		As a by-product, we show some analogous results in purely topological context (without direct use of model theory).
	\end{abstract}
	
	\maketitle
	\section{Introduction}
	The so-called strong types (i.e.\ bounded invariant equivalence relations finer than equivalence over $\emptyset$) and the related concept of connected group components play a very important role in modern model theory. They appear in many influential theorems and conjectures, e.g.\ Pillay's conjecture for groups in o-minimal theories, the independence theorem in simple theories, as well as much of stability theory.
	
	Among these, the type-definable strong types and connected components are particularly well-understood, as they have very tractable topological nature. On the other hand, the non-type-definable ones are much less tame. Indeed, when we try to associate them with topological objects in an analogous manner, we may obtain trivial topologies.
	
	In this paper, we try to make more precise the boundary between well- and ill-behaviour. In \cite{CLPZ01}, the authors suggested that the quotients by the Lascar strong type (which may not be type-definable) can be understood via descriptive set theory. This was put in concrete form in a conjecture by the authors of \cite{KPS13}, later proved in \cite{KMS14}.
	\begin{fct}[{Essentially \cite[Conjecture 1]{KPS13}, \cite[Main Theorem A]{KMS14}}]
		\label{fct:smt_KMS}
		The Lascar strong type $\equiv_L$ on a given type-definable set is type-definable (as an equivalence relation) if and only if it is smooth.\xqed{\lozenge}
	\end{fct}
	(The definition of smoothness will be given in Section~\ref{sec:prelims}.)

	Later work  published in \cite{KM14} and independently in \cite{KR16} yields an analogous fact for $F_\sigma$ orbital equivalence relations (which include $\equiv_L$), under the additional assumption that the domain is the set of realisations of a single complete $\emptyset$-type. Finally, in \cite{KPR15}, the same fact was proven for arbitrary equivalence relations defined on a single complete $\emptyset$-type.
	\begin{fct}[see Fact~\ref{fct:KPR_main} for the precise statement]
		\label{fct:smt_KPR_rough}
		Suppose $p\in S(\emptyset)$ is a type in countably many variables, while $E$ is a  strong type on $p(\fC)$. Then $E$ is type-definable if and only if it is smooth.\xqed{\lozenge}
	\end{fct}
	
	Very roughly speaking, the bulk of the proof of both Facts~\ref{fct:smt_KMS} and \ref{fct:smt_KPR_rough} is to prove that smoothness implies that each class is type-definable (in case of Fact~\ref{fct:smt_KMS}, by careful analysis of the diameters related to the Lascar strong type using descriptive set theoretical methods, and in case of Fact~\ref{fct:smt_KPR_rough}, by using topological dynamical methods, introduced in \cite{KP14}). Once we have that, we need only to apply the following fact.
	\begin{fct}
		\label{fct:KPR_rem}
		Suppose either:
		\begin{enumerate}[label=\alph*)]
			\item
			$E={\equiv_L}$, or
			\item
			$E$ is a strong type defined on $p(\fC)$ for some $p\in S(\emptyset)$.
		\end{enumerate}
		Then if all classes of $E$ are type-definable, then $E$ is type-definable.
	\end{fct}
	\begin{proof}
		If $E={\equiv_L}$, it follows easily from \cite[Corollary 1.8]{Ne03}. If $E$ is defined on $p(\fC)$ for a complete $p$, then it is the content of \cite[Remark 1.11]{KPR15}.
	\end{proof}
	In this paper, the goal is to find a common generalisation of Facts~\ref{fct:smt_KMS},~\ref{fct:smt_KPR_rough} by finding and proving a common generalisation of the two variants of Fact~\ref{fct:KPR_rem}, and then using that in conjunction with Fact~\ref{fct:smt_KPR_rough}. To that end, in Section~\ref{sec:general}, we reintroduce the notion of orbital equivalence relation from \cite{KR16} (in an abstract setting), generalise it to the notion of a weakly orbital equivalence relation, and we prove a broad generalisation of Fact~\ref{fct:KPR_rem} in in the form of Theorem~\ref{thm:worb_aut}  (in case of $\equiv_L$, circumventing the use of \cite{Ne03}). This allows us to deduce the main theorem of this paper.
	\begin{mainthm}[Simplified form of Corollary~\ref{cor:smt_aut}]
		Suppose $T$ is a countable theory and $X$ is a type-definable set (in a countable product of sorts).
		
		Suppose $E$ is a strong type on $X$ which is weakly orbital by type-definable (which includes orbital strong types like $\equiv_L$, and -- if $X=p(\fC)$ for some $p\in S(\emptyset)$ -- all strong types on $X$). Then $E$ is smooth if and only if it is type-definable.\xqed{\lozenge}
	\end{mainthm}

	By very similar reasoning, in Section~\ref{sec:def} we obtain analogous results for definable group actions (in parallel to variants of results in \cite{KMS14}, \cite{KM14}, \cite{KR16} and \cite{KPR15} for definable group components) and, in Section~\ref{sec:cpct}, for continuous actions of compact groups.	
	
	\section{Preliminaries}
	\label{sec:prelims}

	\subsection{Basic facts from model theory}
	\label{ssec:mt_prelim}
	Here we recall very briefly some basic facts and definitions which will be applied in the model-theoretical parts of this paper. This is not comprehensive, and is only meant to remind the most important notions; for more in-depth explanation, see e.g.\ \cite{TZ12} for the elementary and \cite{CLPZ01} for the more advanced topics.

	\begin{dfn}
		Some of the basic definitions and conventions we will use are the following.
		\begin{itemize}
			\item
			We fix a complete first-order theory $T$ with infinite models.
			\item
			By $\fC$ we will denote a monster model of $T$, i.e.\ a model which is $\kappa$-saturated and $\kappa$-strongly homogeneous for a sufficiently large (strong limit) cardinal $\kappa$ (or simply a saturated model of cardinality $\kappa$, for a sufficiently large strongly inaccessible cardinal $\kappa$, if such $\kappa$ exists).
			\item
			We say that something is \emph{small} or \emph{bounded} if it is smaller than the $\kappa$ from the definition of the monster model.
			\item
			We say that a set is $A-$\emph{type-definable} (or \emph{type-definable over $A$}) if it is the intersection of a small number sets, each definable over $A$ in a (fixed) small product of sorts of $\fC$. When $A$ is irrelevant, we simply say that the set is type-definable.
			\item
			$\equiv$ is the relation of having the same type over $\emptyset$.
			\item
			We say that a set (a subset of a small product of sorts of $\fC$) is \emph{invariant} if it is setwise invariant under the standard action of $\Aut(\fC)$ or, equivalently, if it is $\equiv$-invariant. When $A$ is a small set, we say that a set is \emph{invariant over $A$} if it is invariant under the group of automorphisms which fix $A$ pointwise.\xqed{\lozenge}
		\end{itemize}
	\end{dfn}
	
	\begin{dfn}
		Here we list the most relevant definitions related to equivalence relations and connected components in model theory:
		\begin{itemize}
			\item
			We say that an equivalence relation is \emph{bounded} if it has a small number of classes.
			\item
			On every product of sorts of $\fC$, we have $\equiv_{KP}$ and $\equiv_L$ (\emph{Kim-Pillay} and \emph{Lascar equivalence}), which are the finest bounded and, respectively, $\emptyset$-type-definable and invariant equivalence relations. Their classes are called \emph{Kim-Pillay} and \emph{Lascar strong types} (respectively).
			\item
			If $G$ is a type-definable group (i.e.\ $G$ is type-definable and the graph of its group operation is type-definable) while $A$ is a small set, then we have the \emph{connected components} $G^{00}_A$ and $G^{000}_A$, which are the smallest $A$-type-definable and invariant over $A$ (respectively) subgroups of small index in $G$.
			\item
			If $E$ is a bounded invariant (over a small set) equivalence relation on an invariant set $X$, then we have the so-called \emph{logic topology} on $X/E$, where the closed sets are precisely those which have type-definable preimages under the quotient map $X\to X/E$.
			\item
			There are groups $\Autf_{KP}(\fC)$ and $\Autf_{L}(\fC)$ called (respectively) \emph{Kim-Pillay} and \emph{Lascar strong automorphism groups}. They are normal subgroups of $\Aut(\fC)$ such that $\equiv_{KP}$- and $\equiv_L$-classes are exactly the orbits of the appropriate strong automorphism groups.
			\item
			The quotient $\Aut(\fC)/\Autf_L(\fC)$ is called the \emph{Galois group} of the theory $T$ and denoted by $\Gal(T)$. It does not depend on the choice of the monster model.\xqed{\lozenge}
		\end{itemize}
	\end{dfn}
	
	\begin{fct}
		\label{fct:basic_mt_obs}
		Some basic facts and observations:
		\begin{enumerate}
			\item
			\label{it:fct:basic_mt_obs:inv_tdf}
			A type-definable and $\equiv_L$-invariant set is $\equiv_{KP}$-invariant. (see e.g.\ \cite[Corollary 1.13]{KPR15})
			\item
			\label{it:fct:basic_mt_obs:las_model}
			If $M\preceq \fC$ is any submodel, then having the same type over $M$ implies having the same Lascar strong type (i.e.\ being $\equiv_L$-equivalent).
			\item
			\label{it:fct:basic_mt_obs:Gal_model}
			If $m$ is a tuple (infinite) enumerating a small submodel $M\preceq \fC$, then the map $\Aut(\fC)\to [m]_{\equiv}/{\equiv_L}$ defined by $\sigma\mapsto [\sigma(m)]_{\equiv_L}$ factors through $\Gal(T)$, and in fact the induced map $\Gal(T)\to [m]_{\equiv}/{\equiv_L}$ is a bijection. By identifying (via this map) $\Gal(T)$ with $[m]_{\equiv}/{\equiv_L}$, we obtain the logic topology on $\Gal(T)$, which does not depend on the choice of $m$.
			\item
			$\Gal(T)$ is a (possibly non-Hausdorff) compact topological group when endowed with the logic topology.
			\item
			The closure of the identity in $\Gal(T)$ is $\Autf_{KP}(\fC)/\Autf_L(\fC)$. It is sometimes denoted by $\Gal_0(T)$.
			\item
			\label{it:fct:basic_mt_obs:T_2}
			If $X$ is a type-definable set and $E$ is a bounded invariant equivalence relation on $X$, then $X/E$ is a Hausdorff space (with the logic topology) if and only if $E$ is type-definable.\xqed{\lozenge}
		\end{enumerate}
	\end{fct}
	With regard to \eqref{it:fct:basic_mt_obs:T_2}, the fact that type-definable relations have Hausdorff quotients is folklore, while the converse is easy to see, as $E$ is the preimage of the diagonal of $X/E$, which -- provided $X/E$ is Hausdorff -- is a closed subset of $(X/E)^2$ (in the product of logic topologies on $X/E$, and hence in the -- a priori stronger -- logic topology on $X^2/(E\times E)$).

	\subsection{Descriptive set theory}
	\label{ssec:dst_intro}
	To formulate the main results, we will need the classical notion of smoothness from descriptive set theory.
	
	\begin{dfn}
		\label{dfn:smt}
		Suppose that $E$ is an equivalence relation on a Polish space $X$. We say that $E$ is \emph{smooth} if there is a Borel function $f\colon X\to {\bf R}$ such that $x_1\Er x_2$ if and only if $f(x_1)=f(x_2)$.
		
		Equivalently, $E$ is smooth if it admits a countable Borel separating family, i.e.\ there is a family $(B_n)_{n\in {\bf N}}$ of Borel sets such that $x_1 \Er x_2$ if and only if for all $n$ we have $x_1\in B_n\iff x_2\in B_n$. \xqed{\lozenge}
	\end{dfn}
	Smooth equivalence relations are the well-behaved, tame ones. This includes all closed equivalence relations (as well as $G_\delta$).
	\begin{fct}[{\cite[Theorem 3.4.3]{BK96}}]
		\label{fct:clsd_smth}
		Every closed equivalence relation on a Polish space is smooth.\xqed{\lozenge}
	\end{fct}
	
	The reader may consult \cite{BK96} for more in-depth discussion about group actions from the descriptive-set-theoretic perspective, including smoothness.

	\subsection{More recent facts from model theory}
	\label{ssec:previous_results}
	In this subsection, we recall those of the more recent concepts mentioned in the introduction which will be used in the formulations and proofs of the main results (including Fact~\ref{fct:KPR_main}, which we seek to strengthen in this paper). For a more detailed exposition, one may consult one or more of \cite{KPS13}, \cite{KR16}, \cite{KMS14}, \cite{KM14}, or \cite{KPR15}, particularly the preliminary sections.
	
	Let us first define what smoothness means in the model-theoretic context.
	\begin{dfn}
		If $X$ is an invariant set, then we say that it is \emph{countably supported} if it is a subset of a countable product of sorts of $\fC$.\xqed{\lozenge}
	\end{dfn}
	
	\begin{dfn}
		Suppose that the theory is countable. Let $M$ be a countable model, and let $E$ be a countably supported, bounded equivalence relation, invariant over $M$, on an $M$-type-definable set $X$. Denote by $X_M$ the space of types over $M$ of elements of $X$.
		
		Then $X_M$ is a compact Polish space and we have an equivalence relation $E^M$ on $X_M$, the pushforward of $E$ via $x\mapsto \tp(x/M)$ (i.e.\ $\tp(x_1/M)\Er^M \tp(x_2/M)$ if and only if $x_1 \Er x_2$). (This is well-defined essentially by Fact~\ref{fct:basic_mt_obs} \eqref{it:fct:basic_mt_obs:las_model}.)\xqed{\lozenge}
	\end{dfn}
	
	\begin{dfn}
		\label{dfn:smt_modelth}
		We say that $E$ is \emph{smooth} if $E^M$ is smooth, as per Definition~\ref{dfn:smt} (this does not depend on the choice of the countable model $M$). \xqed{\lozenge}
	\end{dfn}
	
	\begin{fct}
		\label{fct:tdf_smt}
		As an immediate corollary of Fact~\ref{fct:clsd_smth}, in a countable theory, a type-definable, countably supported, bounded equivalence relation is smooth.\xqed{\lozenge}
	\end{fct}
	
	With these definitions, we can formulate the full statement of the main result of \cite{KPR15} (which we use to prove its generalisation in Corollary~\ref{cor:smt_aut}, i.e.\ the Main Theorem).
	\begin{fct}[{\cite[Theorem 5.1]{KPR15}}]
		\label{fct:KPR_main}
		We are working in a monster model $\fC$ of a complete, countable theory. Suppose we have:
		\begin{itemize}[nosep]
			\item
			a $\emptyset$-type-definable, countably supported set $X$,
			\item
			a bounded, invariant equivalence relation $E$ on $X$,
			\item
			a type-definable and $E$-saturated set $Y\subseteq X$.
		\end{itemize}
		Then, for every type $p \in X_{\emptyset}$, either $E\restr_{Y\cap p(\fC)}$ is type-definable [in which case -- provided $Y\cap p(\fC)\neq \emptyset$ -- $E\restr_{p(\fC)}$ is type-definable as well, by Fact~\ref{fct:KPR_rem}], or $E\restr_{Y\cap p(\fC)}$ is not smooth.\xqed{\lozenge}
	\end{fct}
	
	The theorem also has a counterpart for definable groups (as did the main results of \cite{KR16,KM14}), which we generalise in Corollary~\ref{cor:smt_def}.
	\begin{fct}[{\cite[Corollary 5.7]{KPR15}}]
		\label{fct:KPR_main_group}
		Assume the language is countable.
		Suppose that $G$ is a definable group and $H\leq G$ subgroup of bounded index, invariant over a countable set $A$. Suppose in addition that $K\geq H$ is a type-definable subgroup of $G$. Then $E_H\restr_{K}$ is smooth if and only if $H$ is type-definable (where $E_H$ is the relation on $G$ of lying in the same left coset of $H$.)\xqed{\lozenge}
	\end{fct}
	
	(Note that in \cite{KPR15}, $G$ and $H$ were invariant over $\emptyset$, but this is an equivalent formulation, as we can add countably many parameters as constants to the language.)
	
	\section{Abstract orbital and weakly orbital equivalence relations}
	\label{sec:general}
	In this section, $G$ is an arbitrary group, while $X$ is a $G$-space, and neither has any additional structure. The goal here is to define and understand orbital and weakly orbital equivalence relations in such a general context. Note that the parts related to orbital equivalence relations in this and later sections can be read mostly independently of the (more technical) parts related to weak orbitality.

	Unless otherwise stated, all the equivalence relations in this paper are assumed to be invariant.
	\begin{dfn}
		A relation $R$ on $X$ is said to be ($G$-)\emph{invariant} if for every $g\in G$ and $x_1,x_2\in X$ we have $x_1\Rr x_2$ if and only if $gx_1\Rr gx_2$.\xqed{\lozenge}
	\end{dfn}
	

	\subsection{Orbital equivalence relations}
	\label{ssec:orbital}
	To every invariant equivalence relation, we can attach a canonical subgroup of $G$, and dually, each subgroup of $G$ gives us an equivalence relation on $X$, namely its orbit equivalence relation.
	\begin{dfn}
		If $E$ is an invariant equivalence relation, then we define $H_E$ as the group of all elements of $G$ which preserve every $E$-class setwise.
	
		If $H\leq G$, we denote by $E_H$ the equivalence relation on $X$ of lying in the same $H$-orbit.\xqed{\lozenge}
	\end{dfn}	
	
	\begin{ex}
		$E_H$ need not be invariant: for example, if $G=X$ is acting on itself by left translations, then $E_H$ (whose classes are just the right cosets of $H$) is invariant if and only if $H$ is a normal subgroup of $G$.\xqed{\lozenge}
	\end{ex}
	
	The orbital equivalence relations -- defined below -- are extremely well-behaved among the invariant equivalence relations. At the same time, in model-theoretic context, this is a very natural class to consider: all the classical strong types (namely the Lascar, Kim-Pillay and Shelah strong types) are orbital.
	
	\begin{dfn}
		An invariant equivalence relation $E$ is said to be \emph{orbital} if there is a subgroup $H$ of $G$ such that $E=E_H$ (i.e.\ $E$ is the relation of lying in the same orbit of $H$).\xqed{\lozenge}
	\end{dfn}
	
	(Note that if $E$ is orbital, then $E\subseteq E_G$, i.e.\ $E$-classes are subsets of $G$-orbits.)
%
	\begin{prop}
		\label{prop:orb_from_group}\leavevmode
		\begin{itemize}
			\item
			\label{rem:orb_norm}
			If $E$ is an invariant equivalence relation on $X$, then $H_E\unlhd G$ and $E_{H_E}\subseteq E$.
			\item
			If, in addition, $E$ is orbital, then $E=E_{H_E}$.
			\item
			If $H\unlhd G$, then $E_H$ is an invariant equivalence relation
			\item
			\label{rem:orb_to_group}
			If $E_H$ is an invariant equivalence relation, we have $H\leq H_{E_H}$.
		\end{itemize}
	\end{prop}
	\begin{proof}
		Straightforward.
%
	\end{proof}

	\begin{ex}
		The action of $\SO(2)$ on $S^1$ is free, and the group is commutative. This implies that the orbital equivalence relations correspond exactly to subgroups of $\SO(2)$ (in fact, because the action is transitive, those are all the invariant equivalence relations).\xqed{\lozenge}
	\end{ex}
	
	\begin{ex}
		\label{ex:3drotations}
		Consider the natural action of $\SO(3)$ on $S^2$. Certainly, the trivial and total relations are both invariant equivalence relations. Moreover, it is not hard to see that the equivalence relation identifying antipodal points is also invariant.
		
		In fact, those three are the only invariant equivalence relations. Among them, only the first two are orbital.\xqed{\lozenge}
	\end{ex}
	
	\subsection{Weakly orbital equivalence relations}
	As hinted at in the introduction, we want to find a generalisation of orbitality which includes equivalence relations invariant under transitive group actions. Here we define such a notion, and in Subsection~\ref{ssec:orb+trans_as_worb} we will see that it does indeed include both cases.

	The following notation is fundamental for this paper.
	\begin{dfn}
		\label{dfn:rsub}
		For arbitrary $H\leq G$ and $\tilde X\subseteq X$, let us denote by $R_{H,\tilde X}$ the relation (which may not be an equivalence relation) defined by
		\[
		x_1\Rr_{H,\tilde X} x_2\iff \exists g\in G \; \exists h\in H \ \ gx_1=hgx_2\in \tilde X.\xqed{\lozenge}
		\]
	\end{dfn}
	
	\begin{prop}
		\label{prop:alt_rsub}
		Note that $R_{H,\tilde X}$ may also be defined as the smallest relation $R$ such that:
		\begin{itemize}
			\item
			$R$ is invariant,
			\item
			for each $\tilde x\in \tilde X$ and $h\in H$ we have $\tilde x\Rr h\tilde x$.
		\end{itemize}
	\end{prop}
	\begin{proof}
		Straightforward.
	\end{proof}

	
	\begin{dfn}
		\label{dfn:worb}
		We say that $E$ is a \emph{weakly orbital} equivalence relation if there is a subset $\tilde X\subseteq X$ and a subgroup $H\leq G$ such that $E=R_{H,\tilde X}$.\xqed{\lozenge}
	\end{dfn}
	
	Note that, as in the orbital case, we always have $R_{H,\tilde X}\subseteq E_G$.
	
	\begin{ex}
		\label{ex:antipode_worb}
		The antipodism equivalence relation from Example~\ref{ex:3drotations} is weakly orbital: choose any point $\tilde x\in S^2$, and then choose a single rotation $\theta\in \SO(3)$ which takes $\tilde x$ to $-\tilde x$. Then $\tilde X:=\{\tilde x\}$ and $H:=\langle \theta\rangle$ witness weak orbitality.\xqed{\lozenge}
	\end{ex}
	
	For further examples of weakly orbital equivalence relations, see Subsection~\ref{ssec:worb_ex}.

%
%
%
%

	There is a useful, explicit description of ``classes'' of $R_{H,\tilde X}$.
	
	\begin{lem}
		\label{lem:worb_class_description}
		Let $R=R_{H,\tilde X}$. For every $x_0\in X$, we have
		\[
		\{x\mid x_0\mathrel R x \}=\bigcup_{g} g^{-1} H g\cdot x_0,
		\]
		where the union runs over $g\in G$ such that $g\cdot x_0\in \tilde X$.
	\end{lem}
	\begin{proof}
		Follows from Proposition~\ref{prop:alt_rsub}.
	\end{proof}

	\begin{dfn}
		\label{dfn:max_witness}
		A \emph{pair of maximal witnesses} (of weak orbitality of $E$) is a pair $(H,\tilde X)$ which witnesses weak orbitality and is maximal with respect to (joint) inclusion. In this case, we say that each of $H$ and $\tilde X$ is a \emph{maximal witness} of orbitality of $E$.\xqed{\lozenge}
	\end{dfn}
	
	The following lemma is, in part, an analogue of the second bullet of Proposition~\ref{prop:orb_from_group}.
	
	\begin{lem}
		\label{lem:worb_maximal}
		Consider $R=R_{H,\tilde X}$. Then:
		\begin{enumerate}
			\item
			\label{it:lem:worb_maximal_set}
			$R=R_{H,\tilde X'}$, where $\tilde{X'}:=\{x\in X \mid \forall h\in H\ \ x\Rr hx\}$, and
			\item
			\label{it:lem:worb_maximal_group}
			$R=R_{H',\tilde X}$, where $H':=\{g\in G \mid \forall \tilde x\in \tilde X \ \ \tilde x \Rr g\tilde x \}$.
		\end{enumerate}
		
		Moreover, for an equivalence relation $E=R_{H,\tilde X}$:
		\begin{itemize}
			\item
			applying the two operations, in either order, yields a maximal pair of witnesses in the sense of Definition~\ref{dfn:max_witness} (in particular, each of $\tilde X'$ and $H'$ is a maximal witness),
			\item
			every maximal witness $\tilde X$ is a union of $E$-classes.
		\end{itemize}
	\end{lem}
	\begin{proof}
		For \eqref{it:lem:worb_maximal_set}, (by the definition of $\tilde X'$) $R$ is an invariant relation such that for all $\tilde x\in \tilde X'$ and $h\in H$ we have $\tilde x \Rr h\tilde x$. Since $R_{H,\tilde X'}$ is, by Proposition~\ref{prop:alt_rsub}, the finest such relation, it follows that $R_{H,\tilde X'}\subseteq R$. On the other hand, $\tilde X\subseteq \tilde X'$, so $R=R_{H,\tilde X}\subseteq R_{H,\tilde X'}$. The proof of \eqref{it:lem:worb_maximal_group} is analogous.
		
		The first bullet of the ``moreover'' part is clear. The second one follows from the fact that $\tilde X'$ is a union of $E$-classes (which is true because $R$ is invariant).
	\end{proof}

	Note that, in contrast to the orbital case, where we have a canonical maximal witness (namely, $H_E$), in the weakly orbital case, the maximal pairs of witnesses are, in general, far from canonical, for instance because for any $g\in G$ we have $R_{H,\tilde X}=R_{gHg^{-1},g\cdot \tilde X}$. In fact, even up to this kind of conjugation, the choice may not be canonical.
	
	\subsection{Orbitality and weak orbitality; transitive actions}
	\label{ssec:orb+trans_as_worb}
	Now, we proceed to show that the class of weakly orbital equivalence does indeed satisfy its stated objective: generalising the orbital equivalence relations and equivalence relations invariant under transitive actions.
	\begin{prop}
		\label{prop:orb_is_worb}

		Every orbital equivalence relation is weakly orbital. In fact, if $E=E_H$ is an invariant equivalence relation, then $E=R_{H,X}(=R_{H_E,X})$.
%
	\end{prop}
	\begin{proof}
		$E_H$ is by definition the finest relation such that for each $x\in X$ and $h\in H$ we have $x\Er_H hx$. If it is also invariant, we have by Proposition~\ref{prop:alt_rsub} that $R_{H,X}=E_H$.
	\end{proof}
	(It is worth noting that a weakly orbital equivalence relation is orbital if and only if its weak orbitality is witnessed by a normal subgroup of $G$, and also if and only if its weak orbitality is witnessed by $\tilde X=X$.)
%
%
%
%
	
	When the action is transitive, every invariant equivalence relation is weakly orbital. This generalises the idea from Example~\ref{ex:antipode_worb}. (Note that it corresponds to the context of Fact~\ref{fct:KPR_main}.)
	
	\begin{prop}
		\label{prop:single_orbit}
		Suppose the action of $G$ is transitive, and $\tilde x \in X$ is arbitrary.	Then for any invariant equivalence relation $E$ we have $E=R_{\Stab_G\{[\tilde x ]_E\},\{\tilde x \}}$ (where $\Stab_G\{[\tilde x ]_E\}$ is the setwise stabiliser of $[\tilde x ]_E$, i.e.\ $\{g\in G\mid \tilde x\Er g\tilde x \}$).
	\end{prop}
	\begin{proof}
		By transitivity, $\Stab_G\{[\tilde x ]_E\}\cdot \tilde x=[\tilde x ]_E$. Proposition follows from Lemma~\ref{lem:worb_class_description}.
%
	\end{proof}
	
	\begin{rem}
		Note that even for transitive actions, $R_{H,\{\tilde x\}}$ is not in general an equivalence relation, as it may fail to be transitive.\xqed{\lozenge}
	\end{rem}
	
	In a way, the cases described in Propositions~\ref{prop:orb_is_worb} and \ref{prop:single_orbit} are orthogonal classes of weakly orbital equivalence relations: in the former, we can take $\tilde X$ to be the entire space, while in the latter, we can take $\tilde X$ singleton. They also roughly correspond to the two cases considered in Fact~\ref{fct:KPR_rem}.
	
	
%

	Given an arbitrary invariant equivalence relation $E$ on $X$, we can attach to each $G$-orbit $G\cdot \tilde x$ with a fixed ``base point'' $\tilde x$ a subgroup $H_{\tilde x}:=\Stab_G\{[\tilde x ]_E\}$, in which case $G\cdot {\tilde x}/E$ is isomorphic (as a $G$-space) with $G/H_{\tilde x}$.
		
	Along with Lemma~\ref{lem:worb_class_description}, this gives an intuitive description of weakly orbital equivalence relations (among the invariant equivalence relations) as those for which we have a set $\tilde X$ which restricts choice of ``base points'', and at the same time ``uniformly'' limits the manner in which the group $H_{\tilde x}$ changes between various orbits: we can only take a union of conjugates of a fixed subgroup (and the conjugates we take depend also on $\tilde X$).
		
	In later sections, we will consider the behaviour of $E$ when $\tilde X$ is type-definable or closed. We can think of that as somehow ``smoothing'' $E$ between the $G$-orbits.

	\subsection{Further examples of weakly orbital equivalence relations}
	\label{ssec:worb_ex}
	In the first example, we define a class of weakly orbital equivalence relations which are not orbital, on spaces $X$ with large $\lvert X/G\rvert$ (so the action is far from transitive).
	\begin{ex}
		Consider the action of $G$ on $X=G^2$ by left translation in the first coordinate only. Let $\tilde X\subseteq G^2$ be the diagonal, and $H$ be any subgroup of $G$. Then $R_{H,\tilde X}$ is a weakly orbital equivalence relation on $X$ whose classes are sets of the form $(g_1g_2^{-1}Hg_2)\times \{g_2\}$. The relation is orbital if and only if $H$ is normal (because the action is free). Meanwhile, $\lvert X/G\rvert=\lvert G\rvert$.\xqed{\lozenge}
	\end{ex}
	(More generally, if we have a family of weakly orbital equivalence relations for $G$ with the same witnessing group $H$, their disjoint union is weakly orbital.)
%
%
%
	
	The second example shows us that we cannot, in general, choose $\tilde X$ as a transversal of $X/G$ (i.e.\ a set intersecting each orbit at precisely one point).
	
	\begin{ex}
		\label{ex:worb_nontrivial}
		Let $F$ be an arbitrary field. Consider the affine group $G=F^3\rtimes \GL_3(F)$, and let $G'$ be a copy of $G$, disjoint from it.
		
		Let $\ell\subseteq F^3$ be a line containing the origin. Choose a plane $\pi\subseteq F^3$ containing $\ell$. Let $E$ be the invariant equivalence relation on $X=G\sqcup G'$ (on which $G$ acts by left translations) which is:
		\begin{itemize}
			\item
			on $G$: $(x_1,g_1)\Er (x_2,g_2)$ whenever $g_1=g_2$ and $x_1-x_2\in g_2\cdot \pi$,
			\item
			on $G'$: $(x_1',g_1')\Er (x_2',g_2')$, whenever $g_1'=g_2'$ and $x_1'-x_2'\in g_2'\cdot \ell$ (slightly abusing the notation).
		\end{itemize}
		
		Put $H=\ell\times \{I\}\leq G$, let $A\subseteq \GL_3(F)$ be such that $A^{-1}\cdot l=\pi$ and $I\in A$, and finally let $\tilde X=(\{0'\}\times \{I'\})\cup (\{0\}\times A)$ (where $0'$ is the neutral element in the vector space component of $G'$).
		
		Then $E$ is weakly orbital, as witnessed by $\tilde X$ and $H$ (to see this, recall Lemma~\ref{lem:worb_class_description} and notice that the $E$-class and $R_{\tilde X,H}$-``class'' of $I$ both are the union of $I+a^{-1}\cdot \ell$ over $a\in A$, while the class of $I'$ is just $I'+\ell$, and then use the fact that both $E$ and $R_{\tilde X,H}$ are invariant).
		
		We will show that $E$ does not have any $\tilde X_1$ witnessing weak orbitality which intersects each of $G$ and $G'$ at exactly one point. Suppose that $E=R_{H_1,\tilde X_1}$ and $\tilde X_1\cap G=\{g\}$, while $\tilde X_1\cap G'=\{g'\}$. Since the action of $G$ on $X$ is free, it follows from Lemma~\ref{lem:worb_class_description} that $[g]_E=H_1g$ and $[g']_E=H_1g'$. This implies that in fact $H_1\leq F^3$, and the first equality implies that $H_1$ is a plane, while the second one implies that it is a line, which is a contradiction.\qed
	\end{ex}

	\section{(Weakly) orbital equivalence relations for automorphism groups}
	\label{sec:aut}
	In this section, we prove the Main Theorem (i.e.\ Corollary~\ref{cor:smt_aut}). As a side result, we show that orbitality and weak orbitality are well-defined model-theoretic properties for bounded invariant equivalence relations (see Corollary~\ref{cor:mtprop}).
	
	Throughout this section, $X$ is a $\emptyset$-type-definable subset of a small product of sorts in $\fC$, while $G$ is just $\Aut(\fC)$. We use the letter $\sigma$ for arbitrary automorphisms (i.e.\ members of $G=\Aut(\fC)$), and $\Gamma$ and $\gamma$ where we would use $H$ and $h$ in the rest of the paper, so as to follow the notation of \cite{KR16}.
	
	Furthermore, we identify $\Gal(T)$ with $[m]_\equiv/{\equiv_L}$ for a tuple $m$ enumerating a small model, as explained in Fact~\ref{fct:basic_mt_obs}\eqref{it:fct:basic_mt_obs:Gal_model}, and we consider the logic topology on various products; for example, a closed set in $\Gal(T)\times X/{\equiv_L}=([m]_\equiv\times X)/({\equiv_L}\times {\equiv_L})$ is a set whose preimage in $[m]_\equiv\times X$ is type-definable. Note that it is \emph{not} a priori the same as being closed in the product of logic topologies on $\Gal(T)$ and $X/{\equiv_L}$ (the product topology might be coarser)! Similarly, the product relation ${\equiv_L}\times{\equiv_L}$ on a product of two invariant sets is usually not the finest bounded invariant equivalence relation on it (so it coarser than $\equiv_L$ on the product).

	\subsection{Preparatory lemmas in the case of the automorphism group action}
	For brevity, let us write $\bar x$, for any $[x]_{\equiv_L} \in X/{\equiv_L}$, as well as $\bar \sigma$ for $\sigma\Autf_L(\fC)\in \Gal(T)$, and $\bar n$ for $[n]_{\equiv_L}\in [m]_\equiv/{\equiv_L}$. By Fact~\ref{fct:basic_mt_obs}\eqref{it:fct:basic_mt_obs:Gal_model}, we identify $\bar n$ with the sole $\bar{\sigma}\in \Gal(T)$ such that $\bar{\sigma}(\bar m)=\bar n$, where $\bar m=[m]_{\equiv_L}$, and thus also write $\bar n\cdot \bar x$ (meaning $\bar \sigma(\bar x)$).
	\begin{lem}
		\label{lem:agree_aut}
		Let $X$ be a type-definable set, and consider the action of $\Gal(T)$ on $X/{\equiv_L}$. Then:
		\begin{enumerate}
			\item
			\label{it:lem:agree_aut_mult}
			the map $(\bar \sigma,\bar x)\mapsto \bar\sigma(\bar x)$ is continuous,
			\item
			\label{it:lem:agree_aut_diag}
			the map $(\bar\sigma,\bar x_1,\bar x_2)\mapsto (\bar\sigma(\bar x_1),\bar\sigma(\bar x_2))$ is continuous,
			\item
			\label{it:lem:agree_aut_orbit}
			for each $\sigma\in \Aut(\fC)$, the map $\bar x\mapsto (\bar x,\bar \sigma(\bar x))$ is continuous,
			\item
			\label{it:lem:agree_aut_clsd}
			the map $(\bar x,\bar \sigma)\mapsto (\bar x,\bar \sigma(\bar x))$ is closed.
		\end{enumerate}
		(Where we consider the logic topology on $\Gal(T)$, $X/{\equiv_L}$, as well as their products -- note again that the logic topology on the product may be finer than the product of logic topologies.)
	\end{lem}
	\begin{proof}
		Consider the partial type $\Phi(n,x,y)= (mx\equiv ny\land x\in X)$.
		\begin{clm*}
			For any $n\in [m]_\equiv$ and $x,y\in X$, the following are equivalent:
			\begin{itemize}
				\item
				$\bar n\cdot \bar x=\bar y$,
				\item
				$\models (\exists y')\Phi(n,x,y')\land y\equiv_L y'$, and
				\item
				$\models (\exists n')\Phi(n',x,y)\land n\equiv_L n'$.
			\end{itemize}
		\end{clm*}
		\begin{clmproof}
			Suppose $\bar n\cdot \bar x=\bar y$. Then we have some $\sigma\in \Aut(\fC)$ such that $\sigma(\bar m)=\bar n$ and $\sigma(\bar x)=\bar y$. This means that we have some $\tau\in \Autf_L(\fC)$ such that $\tau(\sigma(m))=n$. But $\overline{\tau\circ \sigma}=\bar \sigma$, and taking $y'=\tau\circ\sigma(x)$ gives us the second bullet. For the reverse implication, if $\sigma$ witnesses that $mx\equiv ny'$, then in particular $\bar{\sigma}(\bar m)=\bar{n}$, so by definition $\bar n\cdot \bar x=\bar{\sigma}(\bar x)=\overline{\sigma(x)}=\bar {y'}=\bar y$. The proof that the third bullet is equivalent to the first is analogous.
		\end{clmproof}
		
		It follows that for any $A\subseteq X/{\equiv_L}$ we have $\bar n\cdot \bar x\in A$ if and only if $\models (\exists y)\,\, \bar y\in A\land \Phi(n,x,y)$ (because ``$\bar y\in A$'' is a $\equiv_L$-invariant condition), which is a type-definable condition about $n$ and $x$, if $A$ is closed. This gives us the continuity of $(\bar n,\bar x)\mapsto \bar n\cdot \bar x$, i.e.\ \eqref{it:lem:agree_aut_mult}.
		
		Because the induced maps $(X^2)/{\equiv_L}\to (X/{\equiv_L})^2$ and $\Gal(T)\times (X^2)/{\equiv_L}\to \Gal(T)\times (X/{\equiv_L})^2$ are a (topological) quotient maps, \eqref{it:lem:agree_aut_diag} follows from \eqref{it:lem:agree_aut_mult}, applied to $X^2$.
		
		To obtain the continuity of $\bar x\mapsto (\bar x,\bar n\cdot \bar x)$ for all $\bar n$, note that if we fix any $\bar n\in \Gal(T)$ and some $A\subseteq (X\times X)/({\equiv_L}\times {\equiv_L})$, then likewise $(\bar x,\bar n\cdot \bar x)\in A$ exactly when $\models(\exists y)\,\, (\bar x,\bar y)\in A\land \Phi(n,x,y)$, which is again a type-definable condition about $x$ whenever $A$ is closed.
		
		Similarly, for the closedness of $(\bar n,\bar x)\mapsto (\bar x,\bar n\cdot \bar x)$, note that $(\bar x,\bar y)$ is in the image of $A\subseteq ([m]_{\equiv}\times X)/({\equiv_L}\times {\equiv_L})$ exactly when $\models (\exists n)\,\, (\bar n,\bar x)\in A\land \Phi(n,x,y)$, which is a type-definable condition about $x$ and $y$, as long as $A$ is closed.
	\end{proof}

	\begin{prop}
		\label{prop:orb_to_gal}
		Suppose $E$ is a bounded invariant equivalence relation on an invariant set $X$. Suppose also that $\Gamma\leq \Aut(\fC)$ contains $\Autf_L(\fC)$ [and suppose  $\tilde X\subseteq X$]. Write $\bar{\Gamma}:=\Gamma/\Autf_L(\fC)$ [and $\bar{\tilde X}:=\tilde X/{\equiv_L}$]. Then the following are equivalent.
		\begin{enumerate}
			\item
			\label{it:prop:orb_to_gal:1}
			$E=E_\Gamma$ [$E=R_{\Gamma,\tilde X}$]
			\item
			\label{it:prop:orb_to_gal:2}
			$\bar E=E_{\bar \Gamma}$ [$\bar E=R_{\bar \Gamma,\bar {\tilde X}}$]
		\end{enumerate}
		In particular, $E$ is [weakly] orbital if and only if $\bar E$ is (because we can always assume that $\Autf_L(\fC)$ is contained in the group witnessing [weak] orbitality, as $\Autf_L(\fC)$ fixes each class setwise).
	\end{prop}
	\begin{proof}
		Consider the quotient map $q\colon X\to X/{\equiv_L}$. Then we have:
		\begin{equation}
		\label{eq:prop:orb_to_gal:1}
		\tag{$\dagger$}
		q[\Gamma\cdot x]=\bar{\Gamma}\cdot \bar x,
		\end{equation}
		and, since $\Gamma\cdot x$ is an $\equiv_L$-invariant set, conversely:
		\begin{equation}
		\label{eq:prop:orb_to_gal:2}
		\tag{$\dagger\dagger$}
		\Gamma\cdot x=q^{-1}[\bar{\Gamma}\cdot \bar x].
		\end{equation}

		For the orbital case, the implication \eqref{it:prop:orb_to_gal:1}$\Rightarrow$\eqref{it:prop:orb_to_gal:2} is an immediate consequence of \eqref{eq:prop:orb_to_gal:1} -- $\bar E$-classes are the $q$-images of $E$-classes, while $\bar{\Gamma}$-orbits are the $q$-images of $\Gamma$-orbits. The converse is analogous, as by \eqref{eq:prop:orb_to_gal:2}, $\Gamma$-orbits are the $q$-preimages of $\bar \Gamma$-orbits and of course $E$-classes are $q$-preimages of $\bar E$-classes.
		
		The weakly orbital case can be proved similarly: by Lemma~\ref{lem:worb_maximal}, we can assume that
		\begin{equation}
			\label{eq:prop:orb_to_gal_3}
			\tag{$*$}
			\tilde X=\Autf_L(\fC)\cdot\tilde X,
		\end{equation}
		Then we can just apply Lemma~\ref{lem:worb_class_description}: if
		\[
		[\bar x]_{\bar E}=\bigcup_{\bar \sigma}\bar \sigma^{-1}[\bar \Gamma\cdot \bar\sigma(\bar x)]
		\]
		(where the union runs over $\bar \sigma$ such that $\bar \sigma(\bar x)\in \bar{\tilde X}$), then also
		\[
		[x]_E=q^{-1}[[\bar x]_{\bar E}]=\bigcup_{\bar \sigma}q^{-1}[\bar \sigma^{-1}[\bar \Gamma\cdot \bar\sigma(\bar x)]]=\bigcup_{\sigma}\sigma^{-1}[\Gamma\cdot \sigma(x)],
		\]
		where the last union runs over $\sigma$ such that $\sigma(x)\in \tilde X$. To see the last equality, just note that (by \eqref{eq:prop:orb_to_gal_3}) $\sigma(x)\in \tilde X$ if and only if $\bar\sigma(\bar x)\in \bar{\tilde X}$. This yields \eqref{it:prop:orb_to_gal:2}$\Rightarrow$\eqref{it:prop:orb_to_gal:1}, and the opposite implication is analogous.
	\end{proof}

	\begin{cor}
		\label{cor:mtprop}
		{}[Weak] orbitality of a bounded invariant equivalence relation is a model-theoretic property, i.e.\ it does not depend on the choice of the monster model.
	\end{cor}
	\begin{proof}
		$X/{\equiv_L}$, $\Gal(T)$, $\bar E$ and the action of $\Gal(T)$ on $X/{\equiv_L}$ do not depend on the monster model, so the result is immediate from Proposition~\ref{prop:orb_to_gal}.
	\end{proof}
	
	(See also Question~\ref{qu:orbital_mt}.)

	\begin{lem}
		\label{lem:closed_cl_to_closed_witn}
		If $E=R_{\Gamma,\tilde X}$ is bounded invariant and either:
		\begin{itemize}
			\item
			for each $\tilde x\in \tilde X$, $[\tilde x]_E$ is type-definable, or
			\item
			$\Autf_{KP}(\fC)\leq \Gamma$
		\end{itemize}
		then $\tilde X':=\{x\in X\mid \exists \tilde x\in \tilde X \ \ x \equiv_{KP} \tilde x \}$ satisfies $E=R_{\Gamma,\tilde X'}$. (Note that if $\tilde X$ is type-definable, so is $\tilde X'$, and $\Autf_L(\fC)\cdot \tilde X'=\tilde X'$.)
	\end{lem}
	\begin{proof}
		By Lemma~\ref{lem:worb_maximal}\eqref{it:lem:worb_maximal_group} and Fact~\ref{fct:basic_mt_obs}\eqref{it:fct:basic_mt_obs:inv_tdf}, we can assume without loss of generality that $\Autf_{KP}(\fC)\leq \Gamma$. Then by Lemma~\ref{lem:worb_maximal}\eqref{it:lem:worb_maximal_set}, there is a $\Gamma$-invariant (and hence $\equiv_{KP}$-invariant) $\tilde X''\supseteq \tilde X$ such that $E=R_{\Gamma,\tilde X''}=R_{\Gamma,\tilde X}$; the conclusion follows.
%
	\end{proof}

	\subsection{Results in the case of the automorphism group action}
	The following is a weaker form of Theorem~\ref{thm:worb_aut}, but it does not require the machinery of weakly orbital equivalence relations. It can be seen (in part) as a generalisation of the first case of Fact~\ref{fct:KPR_rem}.
	\begin{thm}
		\label{thm:orb_aut}
		Suppose $E$ is a bounded invariant, orbital equivalence relation on $X$. Then the following are equivalent:
		\begin{enumerate}
			\item
			\label{it:thm:orb_aut_Rtd}
			$E$ is type-definable,
			\item
			\label{it:thm:orb_aut_cltd}
			each $E$-class is type-definable,
			\item
			\label{it:thm:orb_aut_Grcls}
			$\Gamma_E$ is the preimage of a closed subgroup of $\Gal(T)$,
			\item
			\label{it:thm:orb_aut_fromclosed}
			$E=E_\Gamma$ for some $\Gamma$ which is the preimage of a closed subgroup of $\Gal(T)$,
			\item
			\label{it:thm:orb_aut_T2}
			$X/E$ is Hausdorff.
		\end{enumerate}
	\end{thm}
	\begin{proof}
		\eqref{it:thm:orb_aut_Rtd} and \eqref{it:thm:orb_aut_T2} are equivalent by Fact~\ref{fct:basic_mt_obs}\eqref{it:fct:basic_mt_obs:T_2}.
		
		\eqref{it:thm:orb_aut_Rtd} trivially implies \eqref{it:thm:orb_aut_cltd}.
		
		Because $\overline{\Gamma_E}=\bar\Gamma_{\bar E}$ is the intersection of stabilisers of classes of $\bar E$,  we obtain \eqref{it:thm:orb_aut_cltd}$\Rightarrow$\eqref{it:thm:orb_aut_Grcls} by Lemma~\ref{lem:agree_aut}\eqref{it:lem:agree_aut_mult} and Proposition~\ref{prop:orb_to_gal}.
		
		\eqref{it:thm:orb_aut_Grcls} trivially implies \eqref{it:thm:orb_aut_fromclosed} (because $E$ is orbital).
		
		If $\Gamma$ is as in \eqref{it:thm:orb_aut_fromclosed}, then by Proposition~\ref{prop:orb_to_gal}, $\bar E$ is the image of $(X/{\equiv_L})\times \bar \Gamma$ via the closed map from Lemma~\ref{lem:agree_aut}\eqref{it:lem:agree_aut_clsd}, so $\bar E$ is closed, and $E$ is type-definable.
	\end{proof}
	To express the full theorem, we will need the following technical definition.
	\begin{dfn}
		\label{dfn:wo/tdf}
		We say that an invariant equivalence relation $E$ is \emph{weakly orbital by type-definable} if there is a type-definable $\tilde X$ witnessing its weak orbitality.\xqed{\lozenge}
	\end{dfn}
	\begin{rem}
		\label{rem:wo/tdf}
		Note that orbital equivalence relations are weakly orbital by type-definable (via $\tilde X=X$), and on the other hand, if $X=p(\fC)$, all invariant equivalence relations are weakly orbital by type-definable (via $\tilde X= \{x_0\}$ for any $x_0\in X$).\xqed{\lozenge}
	\end{rem}
	
	We can now state the generalisation of Fact~\ref{fct:KPR_rem} promised in the introduction. More specifically, the equivalence of \eqref{it:thm:worb_aut:closed} and \eqref{it:thm:worb_aut:closedcl} below generalises it. On the other hand, the equivalence of \eqref{it:thm:worb_aut:closed} and \eqref{it:thm:worb_aut:closedgp} shows us that tameness of $E$ and tameness of witnesses to its weak orbitality are closely tied, which is reinforced in Corollary~\ref{cor:smt_aut}. (See also Remark~\ref{rem:witn_in_theorem}.)
	\begin{thm}
		\label{thm:worb_aut}
		Suppose $E$ is bounded invariant, weakly orbital equivalence relation on $X$. Then the following are equivalent:
		\begin{enumerate}
			\item
			\label{it:thm:worb_aut:closed}
			$E$ is type-definable,
			\item
			\label{it:thm:worb_aut:closedcl}
			each $E$-class is type-definable and $E$ is weakly orbital by type-definable,
			\item
			\label{it:thm:worb_aut:closedgp}
			$E=R_{\Gamma,\tilde X}$ for some type-definable $\tilde X$ and a group $\Gamma\leq \Aut(\fC)$ which is the preimage of a closed subgroup of $\Gal(T)$,
			\item
			\label{it:thm:worb_aut:T2}
			$X/E$ is Hausdorff.
		\end{enumerate}
	\end{thm}
	\begin{proof}
		Points \eqref{it:thm:worb_aut:closed} and \eqref{it:thm:worb_aut:T2} are equivalent by Fact~\ref{fct:basic_mt_obs}\eqref{it:fct:basic_mt_obs:T_2}.
		
		Note that by Proposition~\ref{prop:orb_to_gal}, $\bar E$ is weakly orbital if and only if $E$ is, and using Lemma~\ref{lem:closed_cl_to_closed_witn} each of the points \eqref{it:thm:worb_aut:closed}-\eqref{it:thm:worb_aut:closedgp} 
		is equivalent to its analogue for $\bar E$ (where we replace $\Aut(\fC)$ with $\Gal(T)$ and $X$ with $X/{\equiv_L}$, mapping each automorphism and point to its class, and we substitute ``closed'' for ``type-definable'' as appropriate). This allows us to freely pass back and forth between $\bar E$ and $E$.
		
		\eqref{it:thm:worb_aut:closed} clearly implies that all $\bar E$-classes are closed, and by weak orbitality of $\bar E$, we have some $\bar \Gamma\leq \Aut(\fC)$ and $\bar{\tilde X}\subseteq X$ such that $E=R_{\bar\Gamma,\bar{\tilde X}}$. But then by Lemma~\ref{lem:worb_maximal}, we can replace $\bar {\tilde X}$ with
		\[
			\bar{\tilde X}'=\{\bar x\in \bar X\mid (\forall \bar\gamma\in \bar\Gamma) \; \bar x \mathrel{\bar \Er} \bar\gamma(\bar x) \},
		\]
		which by Lemma~\ref{lem:agree_aut}\eqref{it:lem:agree_aut_orbit} is the intersection of a family of closed sets, and as such closed itself, which yields \eqref{it:thm:worb_aut:closedcl}.
		
		Similarly, to obtain \eqref{it:thm:worb_aut:closedgp} from \eqref{it:thm:worb_aut:closedcl}, apply Lemma~\ref{lem:worb_maximal} once again, this time replacing $\bar \Gamma$ with a $\bar \Gamma'$, and using Lemma~\ref{lem:agree_aut}\eqref{it:lem:agree_aut_mult}.
		
		Assume \eqref{it:thm:worb_aut:closedgp}. By Lemma~\ref{lem:closed_cl_to_closed_witn}, we obtain closed $\bar{\tilde X}$ and $\bar \Gamma\leq \Gal(T)$ such that $\bar E=R_{\bar{\tilde X},\bar{\Gamma}}$. Now note that while $E_{\bar \Gamma}$ may not be invariant, it is a closed equivalence relation (by Lemma~\ref{lem:agree_aut}\eqref{it:lem:agree_aut_clsd}), and we have
		\[
			\bar x_1 \Rr_{\bar{\tilde X},\bar{\Gamma}}\bar x_2\iff \exists\bar\sigma(\bar{\sigma}(\bar x_1)\Er_{\bar \Gamma}\bar{\sigma}(\bar x_2)\land \bar{\sigma}(\bar{x}_1)\in\bar{\tilde X}).
		\]
		It follows from Lemma~\ref{lem:agree_aut}\eqref{it:lem:agree_aut_mult}, \eqref{it:lem:agree_aut_diag} that the expression under $\exists\bar{\sigma}$ defines a closed set. By the definition of the logic topology (and because projections preserve type-definability), projection preserves closedness, which yields \eqref{it:thm:worb_aut:closed}.
%
	\end{proof}
	
	\begin{rem}
		\label{rem:witn_in_theorem}
		It actually follows from the proof of Theorem~\ref{thm:worb_aut} that we can add to the statement a counterpart of Theorem~\ref{thm:orb_aut}\eqref{it:thm:orb_cpct:clsdGp}, namely that all maximal witnesses $\tilde X$, $\Gamma$ are type-definable and the preimage of a closed subgroup of $\Gal(T)$ (respectively). We also have analogous equivalent conditions in Theorem~\ref{thm:worb_def} as well as Theorem~\ref{thm:worb_cpct}.\xqed{\lozenge}
	\end{rem}
	
	The following corollary is the main result of this paper, and a formal statement behind the Main Theorem in the introduction (consult also Remark~\ref{rem:main_general} below to see how it relates to Fact~\ref{fct:KPR_main}).
	
	\begin{cor}
		\label{cor:smt_aut}
		Assume that the theory is countable. Suppose that $E$ is a bounded, invariant, countably supported equivalence relation on $X$. Assume in addition that $E$ is orbital or, more generally, weakly orbital by type-definable. Then the following are equivalent:
		\begin{enumerate}
			\item
			\label{it:cor:smt_aut:clsd}
			$E$ is type-definable,
			\item
			\label{it:cor:smt_aut:clses}
			each $E$-class is type-definable,
			\item
			\label{it:cor:smt_aut:smt}
			$E$ is smooth,
			\item
			\label{it:cor:smt_aut:T2}
			$X/E$ is Hausdorff,
			\item
			\label{it:cor:smt_aut:rest}
			for each complete $\emptyset$-type $p\vdash X$, the restriction $E\restr_{p(\fC)}$ is type-definable.
		\end{enumerate}
	\end{cor}
	\begin{proof}
		Clearly, \eqref{it:cor:smt_aut:clsd} implies \eqref{it:cor:smt_aut:rest}, which implies \eqref{it:cor:smt_aut:clses}.
		
		By Theorem~\ref{thm:orb_aut} or \ref{thm:worb_aut}, \eqref{it:cor:smt_aut:clses} implies \eqref{it:cor:smt_aut:clsd}.
		
		\eqref{it:cor:smt_aut:clsd} implies \eqref{it:cor:smt_aut:smt} by Fact~\ref{fct:tdf_smt}, and it is equivalent to \eqref{it:cor:smt_aut:T2} by Fact~\ref{fct:basic_mt_obs}\eqref{it:fct:basic_mt_obs:T_2}.
		
		Finally, \eqref{it:cor:smt_aut:smt} implies that each restriction $E\restr_{p(\fC)}$ is smooth, which -- by Fact~\ref{fct:KPR_main} -- implies \eqref{it:cor:smt_aut:rest}.
	\end{proof}
	
	\begin{rem}
		\label{rem:main_general}
		In the context of Corollary~\ref{cor:smt_aut}, Fact~\ref{fct:KPR_main} implies that if $Y\subseteq X$ is a type-definable and $E$-invariant subset of $X$ such that $\Aut(\fC)\cdot Y=X$ and $E\restr Y$ is smooth, then the condition \eqref{it:cor:smt_aut:rest} from the corollary is satisfied.
		
		Using this and Remark~\ref{rem:wo/tdf}, we can see that Corollary~\ref{cor:smt_aut} extends Fact~\ref{fct:KPR_main}. (However, bear in mind that Fact~\ref{fct:KPR_main} is essential in the proof of Corollary~\ref{cor:smt_aut}: it merely hides the heavy topological dynamical machinery used in its proof, which we certainly do not circumvent here.)\xqed{\lozenge}
	\end{rem}

	\section{Type-definable group actions}
	\label{sec:def}
	In this section, $G$ is a type-definable group, $X$ is a type-definable set, while the action of $G$ on $X$ is also type-definable (in the sense that it has a type-definable graph), all in the monster model $\fC$. We will adapt the ideas from the preceding section to obtain parallel theorems in case of definable and type-definable group actions (thus generalising Fact~\ref{fct:KPR_main_group}, just like the previous section generalised Fact~\ref{fct:KPR_main}).
	\subsection{Preparatory lemmas in the case of type-definable group actions}
	
	\begin{prop}
		The analogue of Lemma~\ref{lem:agree_aut} is true for a type-definable group action. More precisely, if $G$ acts type-definably on $X$ (i.e.\ the graph of the action is type-definable), then:
		\begin{enumerate}
			\item
			the map $(g,x)\mapsto g\cdot x$ is pseudo-continuous (preimages of type-definable sets are type-definable),
			\item
			the map $(g,x_1,x_2)\mapsto (g\cdot x_1,g\cdot x_2)$ is pseudo-continuous,
			\item
			for each $g\in G$, the map $x\mapsto (x,g\cdot x)$ is pseudo-continuous,
			\item
			the map $(x,g)\mapsto (x,g\cdot x)$ is pseudo-closed (images of type-definable sets are type-definable).
		\end{enumerate}
	\end{prop}
	\begin{proof}
		Trivial.
	\end{proof}
	
	We need to slightly modify the statement of Fact~\ref{fct:KPR_main_group} to one which fits better in our context.
	\begin{lem}
		\label{lem:smt_transdef}
		Assume that the language is countable. Suppose we have a definable action of $G$ on $X$ (i.e.\ $G$ is a definable group and the action of $G$ on $X$ has a definable graph) which is transitive. Let $E$ be a bounded, $G$-invariant equivalence relation on $X$, which is also invariant over a countable set $A\subseteq \fC$ (i.e.\ $\Aut(\fC/A)$-invariant). Then $E$ is smooth if and only if it is type-definable.
	\end{lem}
	\begin{proof}		
		Choose any point $\tilde x\in X$. Then $H=\Stab_G(\{[\tilde x]_E \})$ is invariant over $A\cup \{\tilde x \}$. Consider the map $\Psi\colon G\to X$ defined by $g\mapsto g\cdot \tilde x$. Then $\Psi^{-1}[E]=R_{H,\{e\}}$ is a left-invariant equivalence relation on $G$: the relation of lying in the same left coset of $H$. Let us denote it by $E'$. It is easy to see that $E$ is type-definable if and only if $E'$ is.
		
		Moreover, $E$ is smooth if and only if $E'$ is. To see that, let $M$ be a countable model containing all the necessary parameters (i.e.\ $\tilde x$, $A$, the parameters in the definitions). Notice that because $\Psi$ is an $M$-definable surjection, it induces a continuous surjection $\Psi_M\colon G_M\to X_M$, where $G_M$ is the set of types over $M$ of elements of $G$ and similarly, $X_M$ are the types over $M$ of elements of $X$. A continuous surjection between compact Polish spaces always has a Borel section, so we can use $\Psi_M$ and its section to transport the Borel separating families (from the definition of smoothness) between $G_M$ and $X_M$.
		
		But we know from Fact~\ref{fct:KPR_main_group} that $E'$ is type-definable if and only if it is smooth, so the same is true for $E$.
	\end{proof}
	
	\begin{rem}
		In Lemma~\ref{lem:smt_transdef}, we only really need $G$ itself to be definable, and even that only because that is assumed in Fact~\ref{fct:KPR_main_group}. (See also Question~\ref{qu:tdf_group}.) \xqed{\lozenge}
	\end{rem}
	
	\subsection{Results in the case of type-definable group actions}

	\begin{thm}
		\label{thm:orb_def}
		Let $G$ be a type-definable group acting type-definably on a type-definable set $X$.
		
		Suppose $E$ is an orbital, $G$-invariant equivalence relation on $X$ with $G^{000}_A$-invariant classes (for some small set $A$). Then the following are equivalent:
		\begin{enumerate}
			\item
			$E$ is type-definable,
			\item
			\label{it:thm:orb_def:clsdcls}
			each $E$-class is type-definable,
			\item
			\label{it:thm:orb_def:clsdgp}
			$H_E$ is type-definable,
			\item
			there is a type-definable subgroup $H\leq G$ such that $E=E_H$.
		\end{enumerate}
		In addition, if $E$ is bounded (equivalently, if $X/G$ is small), then the conditions are equivalent to the statement that $X/E$ is Hausdorff with the logic topology.		
	\end{thm}
	\begin{proof}
		The proof is largely analogous to the proof of Theorem~\ref{thm:orb_aut}. For the most part it is easier, and the only additional difficulty lies in the implication \eqref{it:thm:orb_def:clsdcls}$\Rightarrow$\eqref{it:thm:orb_def:clsdgp}: we have that
		\[
			H_E=\bigcap_{x\in X}\{g\in G\mid gx\Er x \},
		\]
		and each of the sets intersected on the right hand side is type-definable, but the intersection is not small. However, by the assumption, each of the sets on the right hand side contains $G^{000}_A$, so we can replace the large intersection by one of size at most $[G:G^{000}_A]$.
	\end{proof}
	In this context, we have a definition parallel to Definition~\ref{dfn:wo/tdf}, and the analogue of Remark~\ref{rem:wo/tdf} clearly applies.
	\begin{dfn}
		We say that a $G$-invariant equivalence relation $E$ is \emph{weakly orbital by type-definable} if there is a type-definable $\tilde X$ witnessing its weak orbitality.
	\end{dfn}
	
	\begin{thm}
		\label{thm:worb_def}
		In context of Theorem~\ref{thm:orb_def}, if we assume that $E$ is only weakly orbital, then the following are equivalent:
		\begin{enumerate}
			\item
			\label{it:thm:worb_def:clsd}
			$E$ is type-definable,
			\item
			\label{it:thm:worb_def:clsdcls}
			each $E$-class is type-definable and $E$ is weakly orbital by type-definable,
			\item
			\label{it:thm:worb_def:clsdwitn}
			$E=R_{H,\tilde X}$ for some type-definable $H$ and $\tilde X$,
		\end{enumerate}
		In addition, if $E$ is bounded (equivalently, $X/G$ is bounded), then the conditions are equivalent to statement that $X/E$ is Hausdorff with logic topology.
	\end{thm}
	\begin{proof}
		We work as in the proof of Theorem~\ref{thm:worb_aut}. The main difference is that we do not need to pass to the Galois group, and (as in the proof of Theorem~\ref{thm:orb_def}) we need to be a little bit more careful when we apply Lemma~\ref{lem:worb_maximal}. For example, in proving
		\eqref{it:thm:worb_def:clsdcls} from \eqref{it:thm:worb_def:clsd}, at some point we take
		\[
			\tilde X'=\{x\in X\mid (\forall h\in H)\;x\Er hx  \}=\bigcap_{h\in H}\{x\in X\mid x\Er hx  \}.
		\]
		As before, this is a non-small intersection of type-definable sets, but we can instead intersect over $h$ representing cosets in $H/G^{000}_A$.
	\end{proof}
	
	\begin{cor}
		\label{cor:smt_def}
		Assume that the theory is countable. Suppose $G$ is a definable group acting definably on $X$, while $E$ is a bounded, $G$-invariant equivalence relation on $X$, which is also invariant over a countable set $A\subseteq \fC$. Assume in addition that $E$ is orbital or, more generally, weakly orbital by type-definable. Then the following are equivalent:
		\begin{enumerate}
			\item
			\label{it:cor:smt_def:clsd}
			$E$ is type-definable,
			\item
			\label{it:cor:smt_def:clses}
			each $E$-class is type-definable,
			\item
			\label{it:cor:smt_def:smt}
			$E$ is smooth,
			\item
			\label{it:cor:smt_def:T2}
			$X/E$ is Hausdorff,
			\item
			\label{it:cor:smt_def:rest}
			for each $x\in X$, the restriction $E\restr_{G\cdot x}$ is type-definable.
		\end{enumerate}
	\end{cor}
	\begin{proof}
		Clearly, \eqref{it:cor:smt_def:clsd} implies \eqref{it:cor:smt_def:rest}, which implies \eqref{it:cor:smt_def:clses}.
		
		By Theorem~\ref{thm:orb_def} or \ref{thm:worb_def}, \eqref{it:cor:smt_def:clses} implies \eqref{it:cor:smt_def:clsd} (note that the assumptions that $E$ is bounded and $A$-invariant imply together that $E$-classes are $G^{000}_A$-invariant).
		
		\eqref{it:cor:smt_def:clsd} implies \eqref{it:cor:smt_def:smt} by Fact~\ref{fct:tdf_smt}, and it is equivalent to \eqref{it:cor:smt_def:T2} by Fact~\ref{fct:basic_mt_obs}\eqref{it:fct:basic_mt_obs:T_2}.
		
		Finally, \eqref{it:cor:smt_def:smt} implies that each restriction $E\restr_{G\cdot x}$ is smooth, which -- by Lemma~\ref{lem:smt_transdef} -- implies \eqref{it:cor:smt_def:rest}.
	\end{proof}
	
	\begin{rem}
		\label{rem:cor_gener}
		Corollary~\ref{cor:smt_def} can be slightly generalised: if we have a definable group action of $G$ on $X$, a type-definable group $K\leq G$, and some $K$-invariant $Y\subseteq X$, then Corollary~\ref{cor:smt_def} remains true if we replace $G$ and $X$ with $K$ and $Y$ (respectively), which is a direct generalisation of Fact~\ref{fct:KPR_main_group} (where we have $X=G$, $Y=K$ and $E=E_H$).
		
		It follows from an appropriate generalisation of Lemma~\ref{lem:smt_transdef} (with essentially the same proof).\xqed{\lozenge}
	\end{rem}
	
	\section{(Weakly) orbital equivalence relations for compact group actions}
	\label{sec:cpct}
	In this section, $G$ is a compact group acting continuously on $X$. We will prove a counterpart of the main theorem in this (purely topological) context.
%
%
%

	\subsection{Preparatory lemmas in the case of compact group actions}	
	
	\begin{lem}
		\label{lem:agree_proper}
		Let $G$ be a compact Hausdorff group acting continuously on a Hausdorff space $X$. Then we have analogous properties as in Lemma~\ref{lem:agree_aut}, i.e.\ :
		\begin{enumerate}
			\item
			\label{it:lem:agree_proper_mult}
			the map $(g,x)\mapsto g\cdot x$ is continuous,
			\item
			the map $(g,x_1,x_2)\mapsto (g\cdot x_1,g\cdot x_2)$ is continuous,
			\item
			\label{it:lem:agree_proper_orbit}
			for each $g\in G$, the map $x\mapsto (x,g\cdot x)$ is continuous,
			\item
			\label{it:lem:agree_proper_clsd}
			the map $(x,g)\mapsto (x,g\cdot x)$ is closed.
		\end{enumerate}
	\end{lem}
	\begin{proof}
		For \eqref{it:lem:agree_proper_clsd}, notice that the image of $A\subseteq X\times G$ is the projection along $G$ of $(A\times X)\cap \Gamma\subseteq X\times G\times X$, where $\Gamma$ is the graph of the group action, so if $A$ is closed, its image is closed by compactness of $G$. The remaining points are trivial by continuity.
	\end{proof}
	
	The following theorem of Miller was one of the main tools used in \cite{KPR15}, and can be seen as a precursor to Fact~\ref{fct:KPR_main}. We will apply it here to obtain an analogous (but much easier to prove) result in the topological case.
	\begin{fct}[{\cite[Theorem 1]{DM77}}]
		\label{fct:miller_thm}
		Suppose that $G$ is a totally nonmeagre [i.e.\ without any closed subsets meagre in themselves] group, and $H\leq G$. If there is a countable family $(E_i)_{i\in {\bf N}}$ of strictly Baire [i.e.\ whose trace on every subspace has the Baire property, relatively to the subspace topology] 
		sets separating right cosets of $H$ [i.e.\ every right coset is exactly the intersection of those $E_i$ which intersect it], then $H$ is closed.\xqed{\lozenge}
	\end{fct}
	
	For our applications, it is important that a compact group is always totally nonmeagre and Borel sets are always strictly Baire. Note also that we can switch right cosets for left cosets in Fact~\ref{fct:miller_thm} just by taking inverses.
	
	\begin{lem}
		\label{lem:smooth_to_closed}
		Suppose $G$ is a totally nonmeagre topological group acting continuously (or even in a Borel way) and transitively on a Polish space $X$. If $E$ is an invariant equivalence relation which is smooth (in the sense of Definition~\ref{dfn:smt}, i.e.\ it has a countable separating family of Borel sets), then for each $x_0\in X$, the stabiliser $\Stab_G\{[x_0]_E \}$ is closed in $G$.
	\end{lem}
	\begin{proof}
		Choose any $x_0\in X$. Put $H:=\Stab_G\{[x_0]_E \}$. We will use Fact~\ref{fct:miller_thm} to show that $H$ is closed.
		
		Consider the map $\Psi\colon G\to X$ defined by $g\mapsto g\cdot x_0$. By invariance of $E$ and the definition of $H$, we have $\Psi(g_1)\Er \Psi(g_2)$ if and only if $g_1[x_0]_E=g_2[x_0]_E$, i.e.\ $g_1H=g_2H$. This implies that the pullback of a countable separating family for $E$ is a countable separating family for left cosets of $H$ in $G$.
		
		Furthermore, since $\Psi$ is continuous, a pullback of a family of Borel sets will be a family of Borel, and hence strictly Baire sets. Together with the conclusion of the preceding paragraph and Fact~\ref{fct:miller_thm}, this implies that $H$ must be closed.		
	\end{proof}

	
	\subsection{Results in the case of compact group actions}	
	\begin{thm}
		\label{thm:orb_cpct}
		Suppose $G$ is a compact Hausdorff topological group acting continuously on a Hausdorff space $X$.
		
		The following are equivalent for an orbital invariant equivalence relation $E$ on $X$:
		\begin{enumerate}
			\item
			\label{it:thm:orb_cpct:clsd}
			$E$ is closed
			\item
			\label{it:thm:orb_cpct:clsdcls}
			each $E$-class is closed,
			\item
			\label{it:thm:orb_cpct:clsdGp}
			$H_E$ is closed,
			\item
			\label{it:thm:orb_cpct:clsdgp}
			$E=E_H$ for a closed subgroup $H\leq G$,
			\item
			\label{it:thm:orb_cpct:T_2}
			$X/E$ is Hausdorff.
		\end{enumerate}
	\end{thm}
	\begin{proof}
		Apart from \eqref{it:thm:orb_cpct:T_2}, we proceed similarly to the proof of Theorem~\ref{thm:orb_aut}, bearing in mind Lemma~\ref{lem:agree_proper}.
		
		\eqref{it:thm:orb_cpct:T_2}$\Rightarrow$\eqref{it:thm:orb_cpct:clsd} is trivial, while the implication from \eqref{it:thm:orb_cpct:clsdgp} to \eqref{it:thm:orb_cpct:T_2} follows from the fact that quotients of Hausdorff spaces by compact Hausdorff group actions are always Hausdorff (see \cite[Propositions 2, 3 in \S 4, Chapter 3]{NB66}).
	\end{proof}
	The following examples show that we cannot drop the assumption that $G$ is compact in Theorem~\ref{thm:orb_cpct}, even if $G$ is otherwise very tame.
	\begin{ex}
		Consider the action of $G={\bf R}$ on a two-dimensional torus $X={\bf R}^2/{\bf Z}^2$ by translations along a line with an irrational slope (e.g.\ $t\cdot [x_1,x_2]=[x_1+t,x_2+ t\sqrt 2]$). Then for $H=G$ the relation $E_G$ has dense (and not closed) orbits, so in particular, $X/E_G$ has trivial topology.\xqed{\lozenge}
	\end{ex}
	
	\begin{ex}
		Consider the action of $G={\bf R}$ on $X={\bf R}^2$ defined by the formula $t\cdot (x,y)= (x+ty,y)$. Then for $H=G$, the classes of $E_G$ are the singletons along the line $y=0$ and horizontal lines at $y\neq 0$, so they are closed, but $E_G$ is not closed.\xqed{\lozenge}
	\end{ex}

%
%
%

	We have a natural analogue of Definition~\ref{dfn:wo/tdf}, and again, the counterpart of Remark~\ref{rem:wo/tdf} applies.
	\begin{dfn}
		We say that a $G$-invariant equivalence relation $E$ is \emph{weakly orbital by closed} if there is a closed $\tilde X$ witnessing weak orbitality of $E$.\xqed{\lozenge}
	\end{dfn}

	\begin{thm}
		\label{thm:worb_cpct}
		Suppose $G$ is a compact Hausdorff group acting continuously on a Hausdorff space $X$.
		Then the following are equivalent for a weakly orbital equivalence relation $E$:
		\begin{enumerate}
			\item
			\label{it:thm:worb_cpct:closed}
			$E$ is closed,
			\item
			each $E$-class is closed and $E$ is weakly orbital by closed,
			\item
			\label{it:thm:worb_cpct:closedgp}
			$E=R_{H,\tilde X}$ for some closed $H$ and $\tilde X$,
		\end{enumerate}
	\end{thm}
	\begin{proof}
%
		We proceed as in the proof of Theorem~\ref{thm:worb_aut} (skipping the parts where we pass to the Galois group), bearing in mind Lemma~\ref{lem:agree_proper}. In case of \eqref{it:thm:worb_cpct:closedgp}$\Rightarrow$\eqref{it:thm:worb_cpct:closed}, we use the fact that projections along compact sets are closed.
	\end{proof}
	
	\begin{rem}
		A counterpart of Theorem~\ref{thm:orb_cpct}\eqref{it:thm:orb_cpct:T_2} (i.e.\ Hausdorffness of the quotient space) is missing in Theorem~\ref{thm:worb_cpct}. It can be shown under some additional assumptions -- e.g.\ if $E$ has a closed transversal -- but in general the proof seems elusive.\xqed{\lozenge}
	\end{rem}

	\begin{cor}
		\label{cor:smt_cpct}
		Suppose that $G$ is a compact Hausdorff group acting on a Polish space $X$. Suppose that $E$ is an invariant equivalence relation on $X$ which is orbital or, more generally, weakly orbital by closed.
		Then the following are equivalent:
		\begin{enumerate}
			\item
			\label{it:cor:smt_cpct:clsd}
			$E$ is closed,
			\item
			\label{it:cor:smt_cpct:clses}
			each $E$-class is closed,
			\item
			\label{it:cor:smt_cpct:smt}
			$E$ is smooth,
			\item
			\label{it:cor:smt_cpct:rest}
			for each $x\in X$, the restriction $E\restr_{G\cdot x}$ is closed.
		\end{enumerate}
	\end{cor}
	\begin{proof}
		Clearly, \eqref{it:cor:smt_cpct:clsd} implies \eqref{it:cor:smt_cpct:rest}, which implies \eqref{it:cor:smt_cpct:clses}.
		
		By Theorem~\ref{thm:orb_cpct} or \ref{thm:worb_cpct}, \eqref{it:cor:smt_cpct:clses} implies \eqref{it:cor:smt_cpct:clsd}.
		
		By Fact~\ref{fct:clsd_smth}, \eqref{it:cor:smt_cpct:clsd} implies \eqref{it:cor:smt_cpct:smt}.
		
		Finally, \eqref{it:cor:smt_cpct:smt} implies that each restriction $E\restr_{G\cdot x}$ is smooth, which -- by Lemma~\ref{lem:smooth_to_closed} and compactness -- implies \eqref{it:cor:smt_cpct:clses}, since $[x]_E=\Stab_G\{[x]_E \}\cdot x$.
	\end{proof}

	\section{Further problems}
	\label{sec:problems}
	
	\subsection{(Weak) orbitality of unbounded equivalence relations}
	In Corollary~\ref{cor:mtprop}, we have deduced that being orbital or weakly orbital equivalence relation is absolute for a bounded invariant equivalence relation (i.e.\ it does not depend on the choice of the monster model). On the other hand, there are equivalence relations which are unbounded, but are still orbital in every monster model (the simplest example is just equality). Thus, one might wonder whether or not is (weak) orbitality, in general, a well-defined model-theoretic property. The main difficulty seems to be that in the unbounded case, we cannot quantify over automorphisms.
	\begin{qu}
		\label{qu:orbital_mt}
		Can \emph{unbounded} invariant equivalence relation be ``accidentally'' [weakly] orbital? In other words, are there monster models $\fC_1,\fC_2$ and an (unbounded) invariant equivalence relation $E$ such that $E(\fC_1)$ is [weakly] orbital, while $E(\fC_2)$ is not?\xqed{\lozenge}
	\end{qu}
	
	\subsection{Type-definable group actions}
	Except for Corollary~\ref{cor:smt_def} and Lemma~\ref{lem:smt_transdef}, the results of Section~\ref{sec:def} apply equally in the case of a type-definable (not necessarily definable) group action. The main reason we require definability there is that those two were based on Fact~\ref{fct:KPR_main_group} (which assumes definability of the group in its hypotheses), but otherwise, it seems like it should be possible to extend them to type-definable group actions.
	\begin{qu}
		\label{qu:tdf_group}
		Can Corollary~\ref{cor:smt_def} and Lemma~\ref{lem:smt_transdef} be extended to type-definable group actions?\xqed{\lozenge}
	\end{qu}

	\subsection{Further ties to descriptive set theory}
	In Corollaries~\ref{cor:smt_cpct}, \ref{cor:smt_def}, and \ref{cor:smt_aut}, we have tied information about the witnesses of weak orbitality of an equivalence relation to its descriptive-set-theoretic property: smoothness. One might wonder whether this can be used further, for example to determine the Borel cardinality of $X/E$ in the non-smooth case.
	\begin{qu}
		Can we extract any additional set-theoretical or topological information about an invariant equivalence relation $E$ from the witnesses to its weak orbitality?\xqed{\lozenge}
	\end{qu}

	\section*{Acknowledgements}
	I would like to thank my advisor, Krzysztof Krupiński, for his support and helpful feedback. I would also like to thank Ludomir Newelski for the post-seminar discussions, including his suggestions on how to improve the introductory part of this paper.
	
	\printbibliography
\end{document}